\documentclass[a4paper,reqno]{amsart}
\numberwithin{equation}{section}

\newtheorem{theorem}{Theorem}[section]

\newtheorem{lemma}[theorem]{Lemma}

\def\beq{\begin{equation}}
\def\eeq{\end{equation}}
\def\be{\begin{equation*}}
\def\ee{\end{equation*}}
\begin{document}

\begin{centerline}{\bf SURJECTIVELY RIGID CHAINS}\end{centerline}

\vspace{.2in}

\begin{centerline}{M. Montalvo-Ballesteros and J.K. Truss}\end{centerline}

\vspace{.2in}

\begin{centerline}{Department of Pure Mathematics,
          University of Leeds,
          Leeds LS2 9JT, UK}\end{centerline}
          
\begin{centerline}{email pmtjkt@leeds.ac.uk}\end{centerline}
  
\newpage

\title{Surjectively rigid chains}

\author{M. Montalvo-Ballesteros and J.K. Truss}
  \address{Department of Pure Mathematics,
          University of Leeds,
          Leeds LS2 9JT, UK}
  \email{pmtjkt@leeds.ac.uk}

\keywords{linear order, rigid  \\
This paper is based on part of the first author's PhD thesis at the University of Leeds}


\begin{abstract}
We study rigidity properties of linearly ordered sets (chains) under automorphisms, order-embeddings, epimorphisms, and endomorphisms. We focus on two main cases, dense subchains of the real numbers, and 
uncountable dense chains of higher (regular) cardinalities. We also give a Fraenkel--Mostowski model which illustrates the role of the axiom of choice in one of the key proofs. 
\end{abstract}

\maketitle


\section{\textbf{Introduction}}\label{intro}

A classical construction by Dushnik and Miller \cite{Dushnik} shows that there is a dense subchain of the real numbers which is rigid, meaning that its only order-automorphism is the identity map. This is 
constructed by transfinite induction using the fact that the number of `potential' order-automorphisms is $2^{\aleph_0}$. In fact the chain that they construct has the stronger property that it has no 
non-trivial order-embeddings, and in \cite{Droste} it was shown that one can construct dense subchains of $\mathbb R$ which are (automorphism-)rigid, but which nevertheless admit many embeddings. 

It is our purpose in this paper to treat similar questions about epimorphisms, and also to make some remarks about general endomorphisms. To make sense of this, we have to work with the reflexive relation 
$\le$ on a linear order, as any endomorphism of the strict relation $<$ is necessarily injective. An {\em epimorphism} of $(X, \le)$ is then a surjective map $f$ from $X$ to $X$ which preserves $\le$, that 
is, $x \le y \Rightarrow f(x) \le f(y)$. For a general class of maps we would say that $(X, \le)$ is {\em rigid} with respect to that class of maps if the only member of the class which preserves $(X, \le)$ 
is the identity. Thus the classical notion of Dushnik and Miller is automorphism-rigidity, and we get corresponding notions of embedding-rigidity and epimorphism-rigidity.

The strongest possible notion that one could consider is `endomorphism-rigidity', where an endomorphism of $(X, \le)$ is any map preserving $\le$ (now not necessarily injective or surjective). This is 
however an essentially vacuous notion, since any constant map is clearly an endomorphism, and there are $|X|$ of these. So the nearest analogue of the notion of rigidity here is that the only endomorphisms 
which exist are ones which, in a sense to be made precise, are `unavoidable'.

The correct context for these discussions seems to be that of certain monoids which arise naturally in the study of the symmetry properties of chains. The most obvious of these, corresponding to the above 
discussion, are the monoid (group actually) of all automorphisms, Aut$(X, \le)$, and those of all the embeddings, Emb$(X, \le)$, epimorphisms Epi$(X, \le)$, and endomorphisms End$(X, \le)$. We write these 
as $G$, $M$, $S$, and $E$ respectively. An extension of the problem of making one or other of these, but not all, trivial, asks rather what the possible values of these monoids are. This is clearly a lot 
more complicated, but we are able to give connections between these monoids in a few cases. One of these is as follows. By applying the axiom of choice, it is quite easy to see that if $S$ is non-trivial, 
then so is $M$. A stronger version of this would be to show either that there is a monoid embedding of $S$ into $M$, or that $S$ is a homomorphic image of $M$. The use of AC is however quite blatant, and so 
it is not clear whether either of these is true. We are able to show at least that $|S| \le |M|$. Furthermore, we give a model of set theory (without choice) in which $M$ is trivial but $S$ is not.

In the final section we adapt the other part of \cite{Droste}, which treats dense chains of larger cardinalities, and we show that similar techniques, involving stationary sets, can be used to give parallel 
results for epimorphisms in this context.

\section{Preliminary results}
 
\begin{lemma}\label{2.1}
There are $2^{\aleph_0}$ epimorphisms of $({\mathbb N}, \le)$.
\end{lemma} 

\begin{proof} For each $A \subseteq {\mathbb N}$ define $f_A: {\mathbb N} \to {\mathbb N}$ by 

$f_A(2n) = n$

$f_A(2n+1) = \left\{ \begin{array}{lcr} n+1 & \mbox{ if } & n \in A \\
                                         n  & \mbox{ if } & n \not \in A \end{array} \right.$
             
Then $\{f_A: A \subseteq {\mathbb N}\}$ is a family of $2^{\aleph_0}$ distinct epimorphisms of $({\mathbb N}, \le)$.
\end{proof}

\begin{theorem} \label{2.2}  If $(X, \le)$ is a chain having a non-identity epimorphism, then it has at least $2^{\aleph_0}$ epimorphisms (so if $X$ is a dense subchain of $\mathbb R$, $(X, \le)$ has 
exactly $2^{\aleph_0}$ epimorphisms). In fact, {\rm Epi}$({\mathbb N}, \le)$ can be embedded in {\rm Epi}$(X, \le)$, and so {\rm Epi}$(X, \le)$ is not commutative.
\end{theorem}

\begin{proof} Let $f \in {\rm Epi}(X, \le)$ be non-trivial. Then $f$ moves some point $a_1$ say. Let $a_0 = fa_1$, and without loss of generality assume that $a_0 < a_1$. Since $f$ is surjective we may 
choose $a_n \in X$ such that for each $n$, $fa_{n+1} = a_n$. It follows that $a_0 < a_1 < a_ 2< \ldots$ (since if $a_{n+1} \le a_n$ then by applying $f^n$ we would find that $a_1 \le a_0$). 

For each epimorphism $\theta$ of $({\mathbb N}, \le)$ we find a corresponding epimorphism $f_\theta$ of $(X, \le)$. Since 
$f_\theta$ will map $a_n$ to $a_{\theta(n)}$, distinct epimorphisms of $({\mathbb N}, \le)$ give rise to distinct epimorphisms 
of $(X, \le)$. 

Observe that $\theta(n) \le n$, and for each $n$, $\theta(n+1) = \theta(n)$ or $\theta(n) + 1$. Let $f_\theta$ fix all the
points not in $\bigcup_{n \in {\mathbb N}}[a_n, a_{n+1}]$ and let $f_\theta$ map $a_n$ to $a_{\theta(n)}$. Finally we have to 
define $f_\theta$ `in between' on each $[a_n, a_{n+1}]$. This is done by mapping all of $[a_n, a_{n+1}]$ to $\{a_{\theta(n)}\}$
if $\theta(n) = \theta(n+1)$ and $[a_n, a_{n+1}]$ to $[a_{\theta(n)}, a_{\theta(n)+1}]$ by $f^{n-\theta(n)}$ if 
$\theta(n+1) = \theta(n) + 1$.

It is easy to check that this is a monoid embedding of Epi(${\mathbb N}, \le)$ into Epi($X, \le$) (essentially because $f_\theta$ acts on $\{a_n: n \in {\mathbb N}\}$ in precisely the way that $\theta$ acts 
on $\mathbb N$). Hence $|Epi(X, \le)| \ge 2^{\aleph_0}$. One can verify directly that if $X$ is a dense subchain of $\mathbb R$, then $|Epi(X, \le)| \le 2^{\aleph_0}$ (or else one may appeal to the next 
theorem).

For the final sentence it is checked that Epi(${\mathbb N}, \le$) is not commutative. \end{proof}

\begin{theorem} \label{2.3} If $(X, \le)$ is a chain such that {\rm Epi}$(X, \le)$ is non-trivial, then so is {\rm Emb}$(X, \le)$. If $(X, \le)$ is densely ordered, then 
$|{\rm Epi}(X, \le)| \le |{\rm Emb}(X, \le)|$. \end{theorem}

\begin{proof} Let $f \in {\rm Epi}(X, \le)$ be non-trivial. Define $g$ by $g(x) \in f^{-1}x$, which is possible by the axiom of choice. Thus for all $x$, $fg(x) = x$. Hence if $g$ is the identity, so is 
$f$, and it follows that $g$ is non-trivial. To see that $g$ is an embedding, suppose for a contradiction that $a < b$ but $gb \le ga$. Then $fgb \le fga$ which implies that $b \le a$, which is a 
contradiction. 

We now modify this argument in the dense case, and we write $g$ as $\theta(f)$ to indicate its dependence on $f$, and show that $\theta$ is 1--1. For this we have to show that for epimorphisms $f_1$ and 
$f_2$ and an embedding $g$, $f_1g = f_2g = id \Rightarrow f_1 = f_2$. Take any $x \in X$, and let $x_1 = gf_1x$. Then $f_1x_1 = f_1gf_1x = f_1x$. We show that $f_2x = f_1x$. Note that 
$f_1x = f_1x_1 = f_1gf_1x = f_2gf_1x = f_2x_1$, so we just have to show that $f_2x = f_2x_1$. Suppose not, for a contradiction, and without loss of generality, suppose that $f_2x_1 < f_2x$. Since $X$ is 
dense, there is $y$ such that $f_2x_1  < y < f_2x$. Then $x_1 = gf_1x = gf_2x_1 < gy < gf_2x$. If $gy \le x$ then $y = f_1gy \le f_1x = f_1x_1 = f_2x_1$, contrary to $f_2x_1 < y$. Otherwise, $x < gy$, which 
gives $f_2x \le f_2gy = y$, contrary to $y < f_2x$. \end{proof}

\vspace{.1in}

The first part of this proof shows that any epimorphism $f$ has a right inverse $g$, and this is an embedding. The other direction does not work, that is, not every embedding need have a left inverse, and 
indeed, this is a key point in the proof of Theorem \ref{3.1}. For $\mathbb Q$ for instance, we may define an embedding $g$ by letting $g(x) = x$ if $x < \pi$ and $g(x) = x + 1$ if $x > \pi$. If this had a 
left inverse $f$, then for any $a$ and $b$ such that $a < \pi < b$, $g(a) < 4 < g(b)$, and so $a = fg(a) \le f(4) \le fg(b) = b$, so $f(4)$ must lie above every rational $< \pi$ and below every rational 
$> \pi$, so must equal $\pi$, which is not rational. The best we can do is the following.

\begin{lemma} \label{2.4} If $(X, \le)$ is an order-complete chain then any member $g$ of \newline {\rm Emb}$(X, \le)$ whose image is coterminal has a unique left inverse $f$, which is an epimorphism (and 
without the `coterminality' hypothesis, one can show the existence of $f$ provided that also $X$ has endpoints) . \end{lemma}

\begin{proof} Define $\sim$ on $X \setminus {\rm im} \, g$ by letting $x \sim y$ if there is no point of ${\rm im} \, g$ between $x$ and $y$. The $\sim$-classes are then convex, and as the image of $g$ is 
coterminal, for any $\sim$-class $A$, $B = g^{-1}(X^{< A})$ and $C = g^{-1}(X^{> A})$ are non-empty disjoint with union equal to $X$. Since $B$ is closed downwards and $C$ is closed upwards, they are 
semi-infinite intervals, and as $X$ is order-complete, $z =$ sup $B$ = inf $C$ lies in $X$. We let $f$ map all of $A$ to $z$, and similarly for each $\sim$-class, and otherwise, $f$ is $g^{-1}$. Clearly 
there is no other option, so $f$ is unique.

If we relax the coterminality requirement, then $B$ or $C$ may be empty, in which case the whole of $A$ is sent to the least/greatest point of $X$ respectively.  \end{proof}

We remark that in the proof of Theorem \ref{2.2} there are actually a lot more epimorphisms of $(X, \le)$ provided by the map $f$ and the sequence $a_0 < a_1 < a_2 < \ldots$ than those we have described. 
Let $b_0 < b_1 < b_2 < \ldots$ be a sequence in $[a_0, \sup_{m \in {\mathbb N}}a_m)$ such that for each $n$, $f^{k_n}b_{2n+1} = b_{2n}$ for some integer $k_n \ge 1$. Then we can define a corresponding 
endomorphism $g$ of $X$ thus: $g$ fixes all points of $X$ not in $\bigcup_{n \in {\mathbb N}}[b_n, b_{n+1}]$ and otherwise
$$gx = \left\{ \begin{array}{lcr} b_0             & \mbox{ if } & b_0 \le x \le b_1 \\
                                 f^{k_0}x        & \mbox{ if } & b_1 \le x \le b_2 \\
                                 f^{k_0}b_2      & \mbox{ if } & b_2 \le x \le b_3 \\
                                 f^{k_0+k_1}x    & \mbox{ if } & b_3 \le x \le b_4 \\
                                 f^{k_0+k_1}b_4  & \mbox{ if } & b_4 \le x \le b_5 \\
                                                 &  \ldots     &
 \end{array} \right.$$
(Intuitively, each $[b_{2n}, b_{2n+1}]$ is mapped to a singleton, and other points are mapped by a suitable power of $f$ to give continuity. This will be an epimorphism provided that the supremum of
$\{b_0, f^{k_0}b_2, f^{k_0+k_1}b_4, \ldots\}$ equals that of the $a_n$.)

There is a finite version of the same thing, where $b_0 < b_1 < b_2 < \ldots < b_{2N} = \sup_{n \in {\mathbb N}}a_n$ for which $g$ fixes all points of $X$ not in $\bigcup_{0 \le n < 2N}[b_n, b_{n+1}]$, and 
otherwise 
$$gx = \left\{ \begin{array}{lcl} b_0                             & \mbox{ if } & b_0 \le x \le b_1 \\
                                 f^{k_0}x                         & \mbox{ if } & b_1 \le x \le b_2 \\
                                 f^{k_0}b_2                       & \mbox{ if } & b_2 \le x \le b_3 \\
                                 f^{k_0+k_1}x                     & \mbox{ if } & b_3 \le x \le b_4 \\
                                 f^{k_0+k_1}b_4                   & \mbox{ if } & b_4 \le x \le b_5 \\
                                                                  &   \ldots    &             \\
                          f^{k_0+k_1 + \ldots + k_{N-1}}b_{2N-2}  & \mbox{ if } & b_{2N-2} \le x \le b_{2N-1} \\
                          f^{k_0+k_1 + \ldots + k_{N-1}}x         & \mbox{ if } & b_{2N-1} \le x \le b_{2N}                
 \end{array} \right.$$
In this case, the last piece must be a translation, not a constant, and $g$ is automatically an epimorphism, since $f$ fixes $\sup a_n = b_{2N}$.

Let us say that epimorphisms $g$ of this form are generated from $f$ {\em in a wide sense}. More accurately, they are generated from the action of $f$ on $\bigcup_{n \in {\mathbb N}}[a_n, a_{n+1}]$ together 
with a choice of $b_n$ and $k_n$.

Thus if we are trying to construct a dense rigid subchain $X$ of $\mathbb R$ which admits an epimorphism $f$ of the form 
$$fx = \left\{ \begin{array}{lcr} x               & \mbox{ if } &  x \le 0       \\
                                 0               & \mbox{ if } &  0 \le x \le 1 \\
                                 x-1             & \mbox{ if } &  x \ge 1 
 \end{array} \right.$$
the above discussion shows that we {\em cannot} avoid also admitting all epimorphisms generated from $f$ in a wide sense.

Similar remarks apply to embeddings. Consider a non-trivial embedding $f$, and suppose that $a < f(a)$ for some $a$. Let $b_0 < b_1 < b_2 < \ldots$ be a sequence in $[a, \sup_{m \in {\mathbb N}}f^ma)$, and 
let $h$ be a function on $\omega$ such that $f(b_0) \le h(0) \le f^{k_0}(b_0)$, and $f^{k_i}(b_{i+1}) \le h(i+1) \le f^{k_{i+1}}(b_{i+1})$ for each $i$ where $1 \le k_0 \le k_1 \le k_2 \le \ldots$, and 
$\sup_{i \in {\mathbb N}}f^{k_i}(b_i) = \sup_{i \in {\mathbb N}}b_i$. Then we say that a function $g$ which agrees with $f$ except on $\bigcup_{n \in {\mathbb N}}[b_n, b_{n+1}]$ where it is given by 
$g(x) = f^{k_i}(x)$ if $b_i < x < b_{i+1}$ and $g(b_i) = h(i)$ is {\em generated from $f$ in a wide sense}. Since $f$ is an embedding, it follows easily that so is any such $g$. There are $2^{\aleph_0}$ 
such maps arising from the different choices of sequence $(k_i)$ and values of the function $h$. The same idea may be applied if $a > f(a)$ for some $a$.

Still, this is quite a restricted class of maps, and we want to show that $X$ can be constructed so that once we have decided to include $f$ as an epimorphism or embedding of $X$, there are no epimorphisms 
or embeddings apart from ones generated from it in a wide sense.
 
\begin{lemma}\label{2.5} If ${\mathbb Q} \subseteq X \subseteq {\mathbb R}$ then $|{\rm End}(X)| = 2^{\aleph_0}$. \end{lemma} 

\begin{proof} If $f \in {\rm End}(X)$ then $a \in {\mathbb R}$ is a {\em discontinuity} of $f$ if $\lim_{x \to a^-}f(x) < \lim_{x \to a^+}f(x)$. Since $\mathbb Q$ is dense in $\mathbb R$, any $f$ can only 
have countably many discontinuities. For any countable subset $C$ of $\mathbb R$, let 
$$F_C = \{f \in {\rm End}(X): \lim_{x \to a^-}f(x) < \lim_{x \to a^+}f(x) \Rightarrow a \in C\}.$$ 
If $f, g \in F_C$ agree on $(C \cap X) \cup {\mathbb Q}$, then they agree on the whole of $X$. For if $x \in X \setminus (C \cup {\mathbb Q})$, let $x_n$ be a sequence of rationals tending to $x$. 
Then as $f$ and $g$ agree on $(C \cap X) \cup {\mathbb Q}$, $fx_n$ and $gx_n$ are equal, and as $x \not \in C$, $fx = gx$. There are $2^{\aleph_0}$ choices of $C$. Also given a choice of $C$, any 
$f \in F_C$ is continuous at all points of $X \setminus C$, so the restriction of $f$ to $X \setminus C$ is determined by its values on $\mathbb Q$, so there are just $2^{\aleph_0}$ possibilities; also if 
$x \in C$ then $f(x)$ has at most $2^{\aleph_0}$ possible values. Putting these together gives at most $2^{\aleph_0} \times (2^{\aleph_0})^{\aleph_0} = 2^{\aleph_0}$ values in all, and it follows that 
$|{\rm End}(X)| = 2^{\aleph_0}$.  \end{proof} 
 
\begin{lemma}\label{2.6}  If $f: {\mathbb R} \to {\mathbb R}$ is order-preserving with dense image, then it is a continuous epimorphism.  \end{lemma} 

\begin{proof} First we see that $f$ is surjective. If not, there is $x \in {\mathbb R} \setminus {\rm im} \, f$. Let $A = f^{-1}(-\infty, x)$ and $B = f^{-1}(x, \infty)$. Then as $f$ is order-preserving 
with dense image, $A$ and $B$ are non-empty disjoint with union $\mathbb R$ and $A < B$. Hence for some $a \in {\mathbb R}$, $A = (-\infty, a]$ and $B = (a, \infty)$, or 
$A = (-\infty, a)$ and $B = [a, \infty)$, suppose the former. Thus $f(a) < x$. By density of im $f$, there is $y \in {\rm im} \, f$ such that $f(a) < y < x$. Let $f(b) = y$. Then $b \in A$ so $b \le a$, but 
this implies that $y \le f(a)$, a contradiction.

To see that $f$ is continuous, let $(x, y) \subseteq {\mathbb R}$ be an open interval. Then $f^{-1}(x, y)$ is convex since $z_1 \le t \le z_2$ and $z_1, z_2 \in f^{-1}(x, y)$ implies that 
$x < f(z_1) \le f(t) \le f(z_2) < y$, and hence that $t \in f^{-1}(x, y)$. So $f^{-1}(x, y)$ is an interval. As in the first part of the proof we see that it is an open interval.  \end{proof} 

\begin{lemma}\label{2.7} (i) If $f \in {\rm Epi}(X, \le)$ for dense $X \subseteq {\mathbb R}$, then $f$ is the restriction to $X$ of a unique epimorphism of $\mathbb R$.

(ii) If $f \in {\rm Emb}(X, \le)$ for dense $X \subseteq {\mathbb R}$, then $f$ is the restriction
to $X$ of an embedding of $\mathbb R$ (which need not be unique).  \end{lemma} 

\begin{proof} (i) Let $F(x) = \sup\{f(y): y \in X, y \le x\}$. Clearly $F$ extends $f$ and is order-preserving. To see that $F$ is surjective, let $y \in {\mathbb R}$. Let $x = \sup\{z \in X: f(z) \le y\}$. 
By definition of $F$, $F(x) \le y$. If $F(x) < y$ then as $X$ is dense, there is $w \in X$ with $F(x) < w < y$ and as $f$ is surjective on $X$, $w = f(u)$ for some $u \in X$. Thus $f(u) < y$ so $u \le x$ 
which gives $w = f(u) \le F(x)$, a contradiction. For uniqueness, we note by Lemma \ref{2.6} that any epimorphism extending $f$ is continuous, and since any two continuous functions which agree on a dense 
set are equal, $F$ is the unique such extension.

(ii) We use the same definition of the extension as in the first part. To see that this time $F$ is an embedding, let $x < y$. Since $X$ is dense, there are $z, u \in X$ such that $x < z < u < y$. Then 
$F(x) \le f(z) < f(u) \le F(y)$ so $F(x) < F(y)$. The extension doesn't need to be continuous. For instance, if $f(x) = 
\left\{ \begin{array}{lcr} x & \mbox{ if } & x < 0 \\
                                 x + 1           & \mbox{ if } & x > 0  
\end{array} \right.$
where $X = (-\infty, 0) \cup (0, \infty)$, there are infinitely many extensions $F$ to $\mathbb R$, corresponding to the possible values of $F(0)$, which can be any member of $[0, 1]$, none of them 
continuous.  \end{proof} 

\vspace{.1in}

We conclude this section by recalling the construction by Dushnik and Miller \cite{Dushnik}, but modified to obtain an embedding-rigid chain.

\begin{theorem} \label{2.8} There is a dense embedding-rigid subchain $X$ of the real line of cardinality $2^{\aleph_0}$.
\end{theorem}   

\begin{proof}  We remark that by Theorem \ref{2.3} this is also epimorphism-rigid.

Let $\kappa = 2^{\aleph_0}$. Enumerate the family of all non-identity embeddings of $\mathbb R$ as $\{f_\alpha: \alpha < \kappa\}$ (using Lemma \ref{2.5}). We choose 
$X_0 \subseteq X_1 \subseteq X_2 \subseteq \ldots$ and $Y_0 \subseteq Y_1 \subseteq Y_2 \subseteq \ldots$ contained in $\mathbb R$ so that for $\alpha < \kappa$, $|X_\alpha|, |Y_\alpha| < \kappa$, and 
$X_\alpha \cap Y_\alpha = \emptyset$. The idea is that the $X_\alpha$ are approximations to $X$, and the $Y_\alpha$ are sets that will definitely be disjoint from $X$. Let $X_0 = {\mathbb Q}$ and 
$Y_0 = \emptyset$. Assuming that $X_\beta, Y_\beta$ have been chosen for $\beta < \alpha$, we shall let $X_\alpha = \bigcup_{\beta < \alpha}X_\beta \cup \{x_\alpha\}$ and 
$Y_\alpha = \bigcup_{\beta < \alpha}Y_\beta \cup \{f_\alpha x_\alpha\}$ for some carefully chosen $x_\alpha$. To make everything work we have to make sure that 
$x_\alpha, f_\alpha x_\alpha \not \in \bigcup_{\beta < \alpha}X_\beta \cup \bigcup_{\beta < \alpha}Y_\beta$ and $x_\alpha \neq f_\alpha x_\alpha$. Since $f_\alpha$ is not the identity, there is $x$ moved by 
$f_\alpha$. If $x < f_\alpha x$, let $I = (x, f_\alpha x)$, and if $x > f_\alpha x$, let $I = (f_\alpha x, x)$. In each case, as $f_\alpha$ is order-preserving, $I \cap f_\alpha I = \emptyset$. Pick any 
$x_\alpha \in I \setminus (\bigcup_{\beta < \alpha} X_\beta \cup \bigcup_{\beta < \alpha} Y_\beta \cup f_\alpha^{-1}(\bigcup_{\beta < \alpha} X_\beta \cup \bigcup_{\beta < \alpha} Y_\beta))$, which is 
possible since $|I| = 2^{\aleph_0} = \kappa$. Then $f_\alpha x_\alpha \in f_\alpha I$ so $f_\alpha x_\alpha \neq x_\alpha$, and $x_\alpha$ is as desired.

Finally we let $X = \bigcup_{\alpha < \kappa} X_\alpha$. Since $X \supseteq {\mathbb Q}$ it is dense. Hence by Lemma \ref{2.7}(ii), any non-identity embedding of $X$ to itself is the restriction to $X$
of a non-identity embedding of $\mathbb R$, which equals $f_\alpha$ for some $\alpha$. But by construction, $f_\alpha$ does not preserve $X$, so this cannot happen.  \end{proof}

\section{Main results on subchains of $\mathbb R$}

We now move on to our main constructions in this case (dense subchains of $\mathbb R$). In all cases, the chain constructed will be automorphism-rigid. In view of the above discussions, we expect to 
describe the following three possible scenarios:

$G = M = S = \{id\}$, and we discuss what $E$ may then be,

$G = S = \{id\}$, $G \neq M$,

$G = \{id\}$, $S \neq \{id\}$. 

These are covered in Theorems \ref{3.3}, \ref{3.1}, and \ref{3.2} respectively. 

\begin{theorem} \label{3.1} There is a dense epimorphism-rigid subchain $X$ of the real line of cardinality $2^{\aleph_0}$ admitting a non-identity embedding $f$ such that every embedding of $X$ is 
generated from $f$ in a wide sense.  \end{theorem}

\begin{proof} Let $f$ be the embedding of $\mathbb R$ given by 
$fx = \left\{ \begin{array}{lcr} x               & \mbox{ if } & x \le 0 \\
                                 x + 1           & \mbox{ if } & 0 <  x  
\end{array} \right.$
      
We shall construct a dense epimorphism-rigid subchain $X$ of $\mathbb R$ which is fixed setwise by $f$ (so that the restriction of $f$ to $X$ gives an embedding of $X$). For this we have to destroy many 
potential non-trivial epimorphisms. 

We again choose subsets $X_\alpha$, $Y_\alpha$ of $\mathbb R$ of cardinality $< \kappa$ where $\kappa = 2^{\aleph_0}$ such that $X_0 \subseteq X_1 \subseteq X_2 \subseteq \ldots$, 
$Y_0 \subseteq Y_1 \subseteq Y_2 \subseteq \ldots$, and $X_\alpha \cap Y_\alpha = \emptyset$. In addition we shall ensure that each $X_\alpha$ is closed under $f$. Since $X_0$ is dense, so will be $X$, and 
this means that any epimorphism or embedding of $(X, \le)$ is the restriction of such a map on $\mathbb R$.

We start with $X_0 = {\mathbb Q}$ and $Y_0 = \emptyset$. Clearly $X_0$ is closed under $f$ so all the required conditions hold. Let $\{f_\alpha: \alpha < \kappa\}$ be an enumeration of all the non-identity 
epimorphisms and embeddings of $\mathbb R$ which are not generated from $f$ in a wide sense (using Lemma \ref{2.5}).

Now assume that $X_\beta$ and $Y_\beta$ have been chosen for all $\beta < \alpha$, and we need to choose $X_\alpha$ and $Y_\alpha$. This hinges on the careful choice of a point $x_\alpha \in {\mathbb R}$. 
Given that, we let $X_\alpha = \bigcup_{\beta < \alpha}X_\beta \cup \{f^nx_\alpha: n \ge 0\}$ and $Y_\alpha = \bigcup_{\beta < \alpha}Y_\beta \cup \{f_\alpha x_\alpha\}$. Clearly 
$|X_\alpha|, |Y_\alpha| < 2^{\aleph_0}$ and $X_\beta \subseteq X_\alpha$, $Y_\beta \subseteq Y_\alpha$, for all $\beta < \alpha$. Also $X_\alpha$ is closed under $f$. 

Now we see how $x_\alpha$ has to be chosen. The fact that $\bigcup_{\beta < \alpha}X_\beta \cap \bigcup_{\beta < \alpha}Y_\beta = \emptyset$ follows from the induction hypothesis. To ensure disjointness of 
$X_\alpha$ and $Y_\alpha$ it is required that for each $n \ge 0$ and $\beta < \alpha$, $f^nx_\alpha \not \in Y_\beta$, $f_\alpha x_\alpha \not \in X_\beta$, and $f^nx_\alpha \neq f_\alpha x_\alpha$. 

We split into the following two cases.

\vspace{.1in}

\noindent{\bf Case 1}: $f_\alpha$ is a non-identity epimorphism. 

In the first case, suppose that $f_\alpha x < x$ for some $x$. We choose 
$y_\alpha \in (f_\alpha x, x) \setminus (\bigcup_{\beta < \alpha}X_\beta \cup f_\alpha(\bigcup_{0 \le n, \beta < \alpha}f^{-n}Y_\beta))$. This is possible since 
$|(f_\alpha x, x)| = 2^{\aleph_0}$ and $|(\bigcup_{\beta < \alpha}X_\beta \cup f_\alpha(\bigcup_{0 \le n, \beta < \alpha}f^{-n}Y_\beta))| < 2^{\aleph_0}$. Using the surjectivity of $f_\alpha$ we now 
choose $x_\alpha$ such that $f_\alpha x_\alpha = y_\alpha$. Note that $x_\alpha > x > y_\alpha$, since if $x_\alpha \le x$ then $y_\alpha = f_\alpha x_\alpha \le f_\alpha x$, contrary to the choice of 
$y_\alpha$. The facts that $f^nx_\alpha \not \in Y_\beta$ and $f_\alpha x_\alpha \not \in X_\beta$ for $\beta < \alpha$ follow from $y_\alpha \not \in f_\alpha f^{-n}Y_\beta$, and 
$y_\alpha \not \in X_\beta$. To see that $f^n x_\alpha \neq f_\alpha x_\alpha$ we note that $f_\alpha x_\alpha < x_\alpha$ but $f^n x_\alpha \ge x_\alpha$.

If $x < f_\alpha x$ for some $x < 0$, we choose $y_\alpha \in (x, f_\alpha x) \setminus (\bigcup_{\beta < \alpha}X_\beta \cup f_\alpha(\bigcup_{0 \le n, \beta < \alpha}f^{-n}Y_\beta))$, and $x_\alpha$ such
that $f_\alpha x_\alpha = y_\alpha$. This time, $x_\alpha < x < y_\alpha$, since if $x \le x_\alpha$ then $y_\alpha = f_\alpha x_\alpha \ge f_\alpha x$, contrary to the choice of $y_\alpha$. To see that
$f^nx_\alpha \neq f_\alpha x_\alpha$, note that as $x_\alpha < 0$, $f^nx_\alpha = x_\alpha$.

Next suppose that $f_\alpha$ moves no point to the left, and it fixes all points less than 0. As it is not the identity, it maps some point to the right. We write $j(x) = f_\alpha(x) - x$, which we note is 
continuous, since $f_\alpha$ is. As $f_\alpha x \ge x$ for all $x$, and $f_\alpha$ is not the identity, $j(x) \ge 0$ for all $x$, and $j(x) > 0$ for some $x$. Since $j$ is continuous, and is zero for $x < 0$,
by the Intermediate Value Theorem there is $x$ such that $0 < j(x) < 1$, and by replacing $x$ by inf$\{y: j(y) = j(x)\}$, we may also suppose that $j(y) < j(x)$ for all $y < x$. This time we choose 
$y_\alpha \in (x, f_\alpha x) \setminus (\bigcup_{\beta < \alpha}X_\beta \cup f_\alpha(\bigcup_{0 \le n, \beta < \alpha}f^{-n}Y_\beta))$, which is possible since $f_\alpha x = x + j(x) > x$. By surjectivity
of $f_\alpha$, there is $x_\alpha$ such that $f_\alpha x_\alpha = y_\alpha$. Then $x_\alpha < x < y_\alpha$. As in the previous paragraphs, it remains to check that $f^nx_\alpha \neq f_\alpha x_\alpha$. But
$f_\alpha x_\alpha \neq x_\alpha$, which covers the case $n = 0$, and if $n \ge 0$, $f^n x_\alpha \ge x_\alpha + 1$ whereas $f_\alpha x_\alpha = j(x_\alpha) + x_\alpha < x_\alpha + 1$.

\vspace{.1in}

\noindent{\bf Case 2}: $f_\alpha$ is an embedding not generated from $f$ in a wide sense.

We choose a non-empty open interval $I$, and $x_\alpha$ will be taken from \newline $I \setminus (\bigcup_{0 \le n, \beta < \alpha}f^{-n}Y_\beta \cup \bigcup_{\beta < \alpha}f_\alpha^{-1}X_\beta)$. Then 
$f^nx_\alpha \not \in Y_\beta$ and $f_\alpha x_\alpha \not \in X_\beta$ are immediately satisfied, and we have to concentrate on ensuring the third condition, that $f^nx_\alpha \neq f_\alpha x_\alpha$.

First suppose that $f_\alpha x \neq x$ for some $x < 0$. If $f_\alpha x < x$ we let $I = (f_\alpha x, x)$ and we see that provided $x_\alpha$ is chosen in $I$, $f^n x_\alpha = x_\alpha$ for all $n$, and
hence does not equal $f_\alpha x_\alpha$ which is $< f_\alpha x$. Similarly, if $f_\alpha x > x$, we let $I = (x, {\rm min}(0, f_\alpha x))$ (which is non-empty open), and again 
$f^nx_\alpha \neq f_\alpha x_\alpha$.

We now suppose that $f_\alpha x = x$ for all $x < 0$. In the next case, there is a non-empty open interval $I$ such that $j(I) \cap {\mathbb N} = \emptyset$, where again $j(x) = f_\alpha(x) - x$. Hence if 
we choose $x_\alpha \in I \setminus (\bigcup_{0 \le n, \beta < \alpha}f^{-n}Y_\beta \cup \bigcup_{\beta < \alpha}f_\alpha^{-1}X_\beta)$, then for any $n$ there is $m$ ($= n$ or 0) such that 
$f^nx_\alpha = x_\alpha + m \neq x_\alpha + j(x_\alpha) = f_\alpha x_\alpha$.

If this condition is violated, then $j(I) \cap {\mathbb N} \neq \emptyset$ for {\em every} non-empty open interval $I$, and we shall show that $f_\alpha$ must be generated from $f$ in a wide sense, contrary 
to our enumeration which was not meant to include any such maps. First we show that $a \le b \Rightarrow j(a) \le j(b)$. By cutting up $[a, b]$ into intervals of length $< 1$ and verifying this condition on 
each such interval, it suffices to prove this statement under the assumption that $b - a < 1$. The hypothesis ensures that there are arbitrarily small $\varepsilon_1, \varepsilon_2 > 0$ such that 
$a + \varepsilon_1 < b - \varepsilon_2$ and $j(a + \varepsilon_1), j(b - \varepsilon_2) \in {\mathbb N}$. Then
\begin{eqnarray*}
j(b) - j(a)  & = & (j(b - \varepsilon_2) - j(a + \varepsilon_1)) + (j(b) - j(b - \varepsilon_2)) +(j(a + \varepsilon_1) - j(a))  \\
          & = & \hspace{-.05in} (j(b \hspace{-.04in}- \hspace{-.04in}\varepsilon_2) \hspace{-.04in} - \hspace{-.04in}j(a \hspace{-.04in}+ \hspace{-.04in}\varepsilon_1)) \hspace{-.04in}+ \hspace{-.04in}(f_\alpha(b) \hspace{-.04in}- \hspace{-.04in}b - \hspace{-.04in}f_\alpha(b \hspace{-.04in}- \hspace{-.04in}\varepsilon_2) \hspace{-.04in}+ \hspace{-.04in}b \hspace{-.04in}- \hspace{-.04in}\varepsilon_2)) \hspace{-.04in} + \hspace{-.04in} (f_\alpha(a \hspace{-.04in} + \hspace{-.04in} \varepsilon_1) 
\hspace{-.04in} - \hspace{-.04in} a \hspace{-.04in} - \hspace{-.04in} \varepsilon_1 \hspace{-.04in} - \hspace{-.04in} 
f_\alpha(a) \hspace{-.04in} + \hspace{-.04in} a)   \\
          & \ge & (j(b - \varepsilon_2) - j(a + \varepsilon_1)) - \varepsilon_2 -  \varepsilon_1
\end{eqnarray*}   
since $f_\alpha(b- \varepsilon_2) \le f_\alpha(b)$ and $f_\alpha(a + \varepsilon_1) \ge f_\alpha(a)$,

and 
\begin{eqnarray*}
j(b - \varepsilon_2) - j(a + \varepsilon_1)  & = & (f_\alpha(b - \varepsilon_2)  - b + \varepsilon_2) - (f_\alpha(a + \varepsilon_1) - a - \varepsilon_1)\\
          & \ge & - b + a + \varepsilon_2 + \varepsilon_1 > -1
\end{eqnarray*}   
since $f_\alpha(a + \varepsilon_1) \le f_\alpha(b - \varepsilon_2)$ and $b - a < 1$, showing that $j(b - \varepsilon_2) - j(a + \varepsilon_1) \ge 0$ as these are integers. Hence 
$j(b) - j(a) \ge - \varepsilon_2 - \varepsilon_1$. Since $\varepsilon_1$ and $\varepsilon_2$ are arbitrarily small, 
$j(b) - j(a) \ge 0$.

Let $I_n = j^{-1}\{n\}$. Then $a < b, a, b \in I_n \Rightarrow [a, b] \subseteq I_n$. Hence $I_n$ are intervals. Clearly $I_0 < I_1 < I_2 < \ldots$, and $\bigcup_{n \in {\mathbb N}}I_n$ is dense in 
$\mathbb R$. It is possible that some $I_n$ may be empty. This does not however apply to $I_0$, which contains $X \cap (-\infty, 0)$ since we are assuming that $f_\alpha$ fixes $(-\infty, 0)$ pointwise. Also
some $I_n$ for $n > 0$ is non-empty, because $f_\alpha$ is not the identity. 

First suppose that there are infinitely many non-empty $I_n$, and let us enumerate them as $J_0 < J_1 < J_2 < \ldots$, and we write $b_{n-1}$ and $b_n$ for the endpoints of $J_n$ for $n \ge 1$ (so that
$\sup J_0 = b_0$). We also let $J_{n+1} = I_{k_n}$. Thus for each $n$, $(b_n, b_{n+1}) \subseteq J_{n+1} \subseteq [b_n, b_{n+1}]$. Thus for $x \in J_{n+1}$, $f_\alpha(x) = x + k_n$, and it follows that 
$f_\alpha$ is generated from $f$ in a wide sense, so was not meant to have been listed. If there are only finitely many non-empty $I_n$s, and the greatest one is $J_N = (b_{N-1}, \infty)$ (or 
$[b_{N-1}, \infty)$) we choose an unbounded sequence $b_N, b_{N+1}, b_{N+2}, \ldots$ with $b_{N-1} < b_N < b_{N+1} < \ldots$ and replace $J_N$ by 
$(b_{N-1}, b_N) \cup (b_N, b_{N+1}) \cup (b_{N+1}, b_{N+2}) \cup \ldots$ and we again obtain an expression for $f_\alpha$ generated from $f$ is a wide sense (but this time with eventually constant value of $k_i$).    \end{proof} 

\vspace{.1in}

We now set about obtaining the best result we can for epimorphisms. The above discussion shows that if we try to do this, we {\em cannot} avoid also admitting all epimorphisms generated from $f$ in a wide 
sense (as well as many embeddings). Still, this is quite a restricted class of maps, and we want to show that $X$ can be constructed so that there are no epimorphisms apart from this. 

\vspace{.1in}

\begin{theorem} \label{3.2} There is a dense automorphism-rigid subchain $X$ of the real line of cardinality $2^{\aleph_0}$ and a 
non-trivial epimorphism $f$ of $X$ such that every member of ${\rm Epi}(X)$ is generated from $f$ in a wide sense.   \end{theorem} 

\begin{proof} We follow a similar method to Theorem \ref{3.1}, this time adding an epimorphism $f$ given by 
$$fx = \left\{ \begin{array}{lcr} x               & \mbox{ if } & x \le 0 \\
                                 0               & \mbox{ if } & 0 \le x \le 1 \\
                                 x - 1           & \mbox{ if } & 1 <  x  
\end{array} \right.$$
and aim to construct a dense embedding-rigid subchain $X$ of $\mathbb R$ closed under $f$. The previous discussion shows that not only is each $f^n$ an epimorphism of $X$, but so is each map of the form
$\varphi$ where for some sequence $0 \le b_0 < b_1 < b_2 < \ldots$ and positive integers $k_i$, $\varphi$ is generated from $f$ in a `wide sense' as defined above. We wish to arrange things so that $X$ has 
no epimorphisms apart from these. 

Let $\kappa = 2^{\aleph_0}$ and let $\{f_\alpha: \alpha < \kappa\}$ be an enumeration of all the epimorphisms of $\mathbb R$ which are not generated from $f$ in a wide sense. As before we choose two
sequences $X_\alpha$ and $Y_\alpha$ contained in $\mathbb R$ for $\alpha < \kappa$. We start with $X_0 = {\mathbb Q}$ and $Y_0 = \emptyset$. 

Assume inductively that $X_\beta$ and $Y_\beta$ have been chosen for $\beta < \alpha$ where $\alpha < \kappa$, and that $X_\beta \cap Y_\beta = \emptyset$, $|X_\beta|, |Y_\beta| < \kappa$, and 
$fX_\beta = X_\beta$. We aim to choose a point $x_\alpha \neq 0$ such that if we let $y_\alpha = f_\alpha x_\alpha$, $X_\alpha = \bigcup_{\beta < \alpha} X_\beta \cup \{f^n x_\alpha: n \in {\mathbb Z}\}$,
and $Y_\alpha = \bigcup_{\beta < \alpha} Y_\beta \cup \{y_\alpha\}$, then $X_\alpha \cap Y_\alpha = \emptyset$. Note that $f^n x_\alpha$ is a single point (even if $n$ is negative), which lies in 
$(x_\alpha + {\mathbb Z}) \cup \{0\}$, since $f$ is 1--1 except at 0. Also $f$ maps $X_\alpha$ onto itself, and since $x_\alpha \in X_\alpha$, and $f_\alpha x_\alpha = y_\alpha$ will not lie in $X$, it will 
follow that $f_\alpha \not \in {\rm Epi}(X, <)$. As usual we have to ensure that 
$x_\alpha \not \in f_\alpha^{-1} \bigcup_{\beta < \alpha}X_\beta \cup \bigcup_{n \in {\mathbb Z}, \beta < \alpha}f^{-n}Y_\beta$ and $f^nx_\alpha \neq f_\alpha x_\alpha$. 

First suppose that there is a non-empty open interval $I$ whose image $J$ under $f_\alpha$ is also open, and such that for all $x \in I$, $f_\alpha x - x \not \in {\mathbb Z}$. We pick 
$y_\alpha \in J \setminus (\bigcup_{\beta < \alpha}X_\beta \cup f_\alpha(\bigcup_{\beta < \alpha}X_\beta \cup \bigcup_{n \in {\mathbb Z}, \beta < \alpha}f^{-n}Y_\beta))$ (noting that as $0 \not \in Y_\beta$,
$|f^{-n}Y_\beta| = |Y_\beta|$, which guarantees that the set we are choosing from has cardinality $\kappa$), and $x_\alpha \in I$ such that $f_\alpha x_\alpha = y_\alpha$. By choice, 
$x_\alpha \not \in f_\alpha^{-1} \bigcup_{\beta < \alpha}X_\beta \cup \bigcup_{n \in {\mathbb Z}, \beta < \alpha} f^{-n}Y_\beta$, so it remains to check that $f^nx_\alpha \neq f_\alpha x_\alpha$. 
Since $x_\alpha \not \in X_0$, $x_\alpha \neq 0$. If $x_\alpha < 0$ then $f^nx_\alpha = x_\alpha$, which cannot equal $f_\alpha x_\alpha$ since $f_\alpha x_\alpha - x_\alpha \not \in {\mathbb Z}$, so we now 
suppose that $x_\alpha > 0$. If $f^nx_\alpha = 0$, then it certainly cannot equal $f_\alpha x_\alpha = y_\alpha$, as this is chosen outside $X_0$. Otherwise, $f^nx_\alpha = x_\alpha - n$ (whether $n$ is
positive or negative), and so $f^nx_\alpha - x_\alpha \in {\mathbb Z}$, and as $f_\alpha x_\alpha - x_\alpha \not \in {\mathbb Z}$ we again deduce that $f^nx_\alpha \neq f_\alpha x$.

Next suppose that $f_\alpha$ moves some $x < 0$. Since $f_\alpha$ is surjective, there is $y$ such that $f_\alpha y = x$ and clearly if $x < f_\alpha x$ then $y < x$, and if $x > f_\alpha x$ then 
$y > x$. Let $J = (x, f_\alpha x)$ (or $(f_\alpha x, x)$ respectively). In the first case, by continuity of $f_\alpha$, the supremum of the set of points which $f_\alpha$ maps to $x$ is also mapped to $x$,
and similarly the infimum of the set of points which $f_\alpha$ maps to $f_\alpha x$ is also mapped to $f_\alpha x$. So by passing to a suitable open subinterval $I$ of $(y, x)$, we may assume that 
$f_\alpha(I) = J$. Note that $I \cap J = \emptyset$. We choose $y_\alpha \in J$ and then $x_\alpha \in I$ as before such that $f_\alpha x_\alpha = y_\alpha$, and we just have to check that 
$f^n x_\alpha \neq f_\alpha x_\alpha$. This time definitely $x_\alpha < 0$, and so $f^n x_\alpha = x_\alpha$, and this cannot equal $f_\alpha x_\alpha$ by assumption. In the second case, where 
$J = (f_\alpha x, x)$, we similarly find a suitable subinterval $I$ of $(x, y)$ which $f_\alpha$ maps to $J$, and again choose $y_\alpha \in J$ and $x_\alpha$ mapping to it under $f_\alpha$. To check that 
$f^n x_\alpha \neq f_\alpha x_\alpha$, the same argument applies if $x_\alpha < 0$. If however $x_\alpha > 0$ then also $f^n x_\alpha \ge 0$, which cannot equal $y_\alpha$ which is $< x$, hence $< 0$.

Now suppose that intervals $I$ and $J$ as in the first case do not exist, and also that $f_\alpha$ fixes $(-\infty, 0)$ pointwise. First we show that $j$ defined by $j(x) = f_\alpha x - x$ is decreasing, 
that is, $a \le b \Rightarrow j(a) \ge j(b)$. If not, there are $a < b$ such that $j(a) < j(b)$. By the Intermediate Value Theorem, $j$ takes all values in $[j(a), j(b)]$ on $[a, b]$, so by increasing $a$ 
and decreasing $b$ as necessary, we may suppose that for $a \le x \le b$, $j(x) \not \in {\mathbb Z}$. By replacing $a$ and $b$ by sup$\{x \in [a, b]: j(a) = j(x)\}$ and inf$\{x \in [a, b]: j(x) = j(b)\}$ 
respectively, we may also suppose that $j$ maps $I = (a, b)$ onto $(j(a), j(b))$. It follows (using the continuity of $f_\alpha$) that $f_\alpha$ maps $I$ onto $J = (f_\alpha(a), f_\alpha(b))$, and since 
for $x \in I$, $f_\alpha x - x \not \in {\mathbb Z}$, we have $I$ and $J$ as desired. Since we are now assuming that such $I$ and $J$ do not exist, we conclude that $j$ is decreasing.

Now consider any $a < b$ for which $f_\alpha(a) < f_\alpha(b)$, and as in the previous paragraph, by passing to a subinterval we may suppose that $f_\alpha(a, b) = (f_\alpha(a), f_\alpha(b))$. Unless $j$ is
constant on $(a, b)$ with an integer value, we can again pass to a suitable subinterval and find $I$ and $J$ as desired. Since we are assuming that they do not exist, it follows that for any $a < b$ such
that $f_\alpha(a) < f_\alpha(b)$ and $f_\alpha(a, b) = (f_\alpha(a), f_\alpha(b))$, $j$ is constant on $(a, b)$ with integer value, so that for some fixed $n \in {\mathbb Z}$, $f_\alpha(x) = x + n$ for
all $x \in (a, b)$.  

Now consider $j^{-1}\{n\}$ for various (integer) values of $n$. Since $j$ is decreasing, $m < n \Rightarrow j^{-1}\{m\} > j^{-1}\{n\}$. We only list non-empty sets of this form, and as we have just shown 
that $j^{-1}\{0\} \supseteq (-\infty, 0)$, it follows that $j^{-1}\{0\} \neq \emptyset$, and $j^{-1}\{n\} \neq \emptyset \Rightarrow n \le 0$. Listing the non-empty sets of the form $j^{-1}\{n\}$ in
increasing order as $J_n$, there is a strictly increasing sequence of integers $(l_n)$ with $l_0 = 0$ and such that $J_0 < J_1 < J_2 < \ldots$ and $J_n = j^{-1}\{-l_n\}$. Write $J_0 = (-\infty, a_0)$, and
$J_n = (a_{2n-1}, a_{2n})$ for $n > 0$. Thus for $x < a_0$, $f_\alpha x = x$ and for $a_{2n-1} < x < a_{2n}$, $f_\alpha x = x - l_n$. Now $f_\alpha$ is increasing, and continuous, so it follows that 
$a_{2n} < a_{2n + 1}$, and also the above discussion shows that $f_\alpha$ is constant on each interval of the form $(a_{2n}, a_{2n+1})$. By continuity, it is actually constant on the {\em closed} interval
$[a_{2n}, a_{2n+1}]$, and also $f_\alpha x = x - l_n$ for all $x \in [a_{2n-1}, a_{2n}]$, so it follows that 
$$f_\alpha x = \left\{ \begin{array}{lcr} x             & \mbox{ if } & x \le a_0 \\
                                      x - l_n           & \mbox{ if } & a_{2n-1} \le x \le a_{2n} \\
                                     a_{2n} - l_n       & \mbox{ if } & a_{2n} \le x \le a_{2n+1}  
\end{array} \right.$$
and by continuity at $a_{2n+1}$, $a_{2n} - l_n = a_{2n+1} - l_{n+1}$. We can now see that $f_\alpha$ is generated from $f$ in a wide sense. We consider two cases.

In the first case, $J_n$ exists for all $n \ge 0$. Then $f_\alpha$ fixes all points not in $\bigcup_{n \in {\mathbb N}}[a_n, a_{n+1}]$, and otherwise,
$$f_\alpha x = \left\{ \begin{array}{lcr} a_0                 & \mbox{ if } & a_0 \le x \le a_1 \\
                                      x - l_1 = f^{l_1}x      & \mbox{ if } & a_1 \le x \le a_2 \\
                                      a_2 - l_1 = f^{l_1}a_2  & \mbox{ if } & a_2 \le x \le a_3 \\
                                      x - l_2 = f^{l_2}x      & \mbox{ if } & a_3 \le x \le a_4 \\
                                      a_4 - l_2 = f^{l_2}a_4  & \mbox{ if } & a_4 \le x \le a_5 \\
                                      x - l_3 = f^{l_3}x      & \mbox{ if } & a_5 \le x \le a_6 \\
                                                              & \ldots      & 
\end{array} \right.$$
and the expression in the desired form follows on letting $k_0 = l_1$ and $k_n = l_{n+1} - l_n$ for $n > 0$ (and we note that $f^{k_n}a_{2n+1} = a_{2n+1} - k_n = a_{2n} - l_n + l_{n+1} - k_n = a_{2n}$).

In the second case, $J_n$ exists just for $0 \le n \le N$, and a similar calculation applies.
\end{proof}

\vspace{.1in}

For the next result we observe that for any dense subchain $X$ of $\mathbb R$, there are certain kinds of endomorphism which we cannot avoid. For $f \in {\rm End}(X, \le)$, let 
$C_f = \{x \in X: \exists a,b(a < x < b \wedge f \mbox{ is constant on } [a, x] \mbox{ or } [x, b])\}$, and $I_f = \{x: f(x) = x\}$. We say that $f$ is {\em locally identity-constant} if $C_f \cup I_f$ 
is dense in $X$. A typical locally identity-constant endomorphism may be constructed as follows. Let $\mathcal{C}$ be a (necessarily countable) family of non-trivial pairwise disjoint intervals (open, 
closed, or semi-open), and for each $C \in {\mathcal C}$ let $a_C \in C$. Let $I$ be any subset of $X \setminus \bigcup {\mathcal C}$ such that $\bigcup{\mathcal C} \cup I$ is dense in $X$. Then we define 
$f$ by letting $f(x) = a_C$ if $x \in C \in {\mathcal C}$, and $f(x) = x$ if $x \in I$. It is easy to verify that this is so far order-preserving. For instance, if $x \le y$ and 
$x \in C_1 \in {\mathcal C}$, $y \in C_2 \in {\mathcal C}$, then $C_1 \le C_2$, so $a_{C_1} \le a_{C_2}$; and if $x \in C \in {\mathcal C}$, $y \in I$ then $C < y$, so $f(x) = a_C < y = f(y)$. We still have 
to define $f(x)$ for $x \not \in \bigcup {\mathcal C} \cup I$. Note that for all $y < x$ lying in $\bigcup{\mathcal C} \cup I$, $f(y) \le x$, and for all $y > x$ lying in $\bigcup{\mathcal C} \cup I$, 
$f(y) \ge x$, which shows that $[\sup \{f(y): y < x, y \in \bigcup{\mathcal C} \cup I\}, \inf \{f(z): x < z, z \in \bigcup{\mathcal C} \cup I\}] \cap X \neq \emptyset$. So we let $f(x)$ be any member of 
$[\sup \{f(y): y < x, y \in \bigcup{\mathcal C} \cup I\}, \inf \{f(z): x < z, z \in \bigcup{\mathcal C} \cup I\}] \cap X$.

\begin{theorem} \label{3.3} There is a dense subchain $X$ of the real line of cardinality $2^{\aleph_0}$ with trivial embedding and epimorphism monoids, and such that every endomorphism of $X$ is 
locally identity-constant.   \end{theorem} 

\begin{proof} We use the model already described in Theorem \ref{2.8}, but modified if necessary by excluding yet more functions. We showed there that the chain constructed is embedding-rigid, and hence by 
Theorem \ref{2.3} is also epimorphism-rigid. It remains to analyze what the endomorphisms $f$ can be. Using Lemma \ref{2.5} we may enumerate all endomorphisms of $\mathbb R$ that are not locally 
identity-constant. In our standard construction of sets $X_\alpha$ and $Y_\alpha$ we need to show how we can choose $x_\alpha$ in such a way that defining $X_\alpha$ to be 
$\bigcup_{\beta < \alpha}X_\beta \cup \{x_\alpha\}$ and $Y_\alpha$ to be $\bigcup_{\beta < \alpha}Y_\beta \cup \{f_\alpha x_\alpha\}$, $X_\alpha \cap Y_\alpha = \emptyset$. As seen in the proofs of Theorems 
\ref{3.1}, \ref{3.2} if there is a non-empty open interval $I$ with image $J$ under $f_\alpha$, such that $J$ also has cardinality $2^{\aleph_0}$ and is disjoint from $I$, then there is a suitable choice of 
$x_\alpha$ and $y_\alpha$. Namely we would choose $y_\alpha \in J \setminus (\bigcup_{\beta < \alpha}X_\beta \cup f_\alpha \bigcup_{\beta < \alpha}Y_\beta)$, and $x_\alpha \in I$ such that 
$f_\alpha x_\alpha = y_\alpha$, and this is sufficient to guarantee that $x_\alpha \not \in \bigcup_{\beta < \alpha}Y_\beta$ and $f_\alpha x_\alpha \not \in \bigcup_{\beta < \alpha}X_\beta$, and 
$f_\alpha x_\alpha \neq x_\alpha$. So it remains to show that for any not locally identity-constant $f_\alpha$, such $I$ and $J$ exist.

For this, observe that as $f_\alpha$ is not locally identity-constant, there are $x < y$ in $\mathbb R$ such that no point of $(x, y)$ lies in $C_{f_\alpha} \cup I_{f_\alpha}$. Thus $f_\alpha$ is strictly 
increasing, hence 1--1, on $(x, y)$. Also, as $f_\alpha$ is not the identity on $(x, y)$ there is $z \in (x, y)$ such that $f_\alpha z \neq z$. If $z < f_\alpha z$ then 
$(z, f_\alpha z) \cap (f_\alpha z, f_\alpha^2 z) = \emptyset$ (and if $f_\alpha z < z$ then $(f_\alpha z, z) \cap (f_\alpha^2 z, f_\alpha z) = \emptyset$). Since $f_\alpha$ is 1--1 on $(x, y)$ we may let 
$I = (z, \min(y, f_\alpha z))$ (or $(\max(x, f_\alpha z), z)$ in the second case), and $J = f_\alpha I$.
\end{proof}

We remark that all locally identity-constant endomorphisms $f$ constructed as above have the property that $f^3 = f^2$ (they would be idempotent, $f^2 = f$, apart from `exceptional' values which may be 
assigned in the final clause, at points of $X \setminus \bigcup{\mathcal C} \cup I$). However, a modification obtained by allowing $C \in {\mathcal C}$ to be cut into infinitely many pieces, gives a more 
general class of such maps, and for these, all powers of $f$ can be distinct. In the simplest example of this type, there is a strictly increasing unbounded sequence $a_0 < a_1 < a_ 2 < \ldots$ in $X$ such
that $f(x) = a_0$ if $x < a_0$ and $f(x) = a_{n+1}$ if $a_n \le x < a_{n+1}$. This is locally identity-constant, but it has infinite order, since for every $n$, $f^n(a_0) = a_n$. It is easy to construct 
more general examples of this type. This does mean that there are always necessarily endomorphisms of $X$ having infinite order. It may be possible to construct examples of dense subchains of $\mathbb R$ 
with trivial embedding and epimorphism monoids, but admitting endomorphisms which are not locally identity-constant, but we have not so far done so.

\section{Results without the axiom of choice}

The main result in this section is to show that the appeal to the axiom of choice in the proof of Theorem \ref{2.3} is really needed. We start by remarking in Theorem \ref{4.1} on how things work out in 
the most well-known Fraenkel--Mostowski model in this context, the so-called Mostowski `ordered' model $\mathcal N$. For this we start in a model $\mathcal M$ of FMC (`Fraenkel--Mostowski with choice', 
obtained from ZFC by altering the axiom of extensionality to allow the existence of `atoms', non-sets which can be members of sets), so that the set $U$ of atoms is indexed by the rational numbers 
$\mathbb Q$, taking the order-preserving permutations as group $G$, and the normal filter of subgroups generated by the pointwise stabilizers of finite subsets of $U$. If we want to arrange 
embedding-rigidity but not epimorphism-rigidity, a modification to the argument is required as in Theorem \ref{4.2}.

\begin{theorem} \label{4.1} In Mostowski's ordered model, the set $U$ of atoms is embedding and epimorphism rigid, and any endomorphism $f$ is locally identity-constant with just finitely many intervals 
on which it is constant. \end{theorem}

\begin{proof} It suffices to verify the final statement, since any locally identity-constant endomorphism which is also an embedding or an epimorphism must be the identity. Since $f \in {\mathcal N}$, there 
is a finite subset $A = \{a_1, a_2, \ldots, a_n\}$ of $U$ such that (in $\mathcal M$) any automorphism of $(U, <)$ fixing $A$ pointwise also fixes $f$. Thus if $g$ lies in the stabilizer $G_A$ of $A$, 
$gf = fg$, which implies that for any $x \in U$, if $g(x) = x$ then $g(f(x)) = f(g(x)) = f(x)$. If $f(x) \not \in A \cup \{x\}$ then there is $g \in G$ fixing $A \cup \{x\}$ pointwise and moving $f(x)$. 
Since we have just remarked that this does not happen, we deduce that $f(x) \in A \cup \{x\}$. Now $f^{-1}\{a_i\}$ are finitely many intervals, and outside $\bigcup_{i = 1}^nf^{-1}\{a_i\}$ all points are 
fixed. Hence $f$ is locally identity-constant.    \end{proof}

We note that in this model, any locally identity-constant endomorphism necessarily has just finitely many intervals on which it is constant, as follows from the above proof. This also follows from the fact 
that $U$ is `o-amorphous', in the sense discussed in \cite{Creed}, as well as the slightly stronger statement that the intervals of constancy have endpoints in $U \cup \{ \pm \infty\}$. Note that in 
addition, $C_f \cup I_f$ actually equals $U$, and is not merely dense.

\begin{theorem} \label{4.2} It is relatively consistent with ZF set theory that there is a dense chain $(X, \le)$ such that ${\rm Emb}(X, \le)$ is trivial, but ${\rm Epi}(X, \le)$ is not. \end{theorem}

\begin{proof} For this it suffices, by standard set-theoretical techniques, to find a Fraenkel--Mostowski model in which the given statement holds. We let the set of atoms be indexed by the family 
$\mathcal Q$ of all eventually zero sequences of rational numbers. Thus $U = \{u_\sigma: \sigma \in {\mathcal Q}\}$. We linearly order $\mathcal Q$, and hence also $U$, anti-lexicographically. That is, 
$\sigma < \tau$ if there is $i$ such that $\sigma(i) < \tau(i)$ and for all $j > i$, $\sigma(j) = \tau(j)$. If we identify ${\mathbb Q}^n$ with the set of members of $\mathcal Q$ which are zero from the 
$n$th place on, each ${\mathbb Q}^n$ is a convex subset of $\mathcal Q$, and ${\mathcal Q} = \bigcup_{n < \omega}{\mathbb Q}^n$.

Clearly as linearly ordered sets, $\mathcal Q$ and $\mathbb Q$ are isomorphic. We let $f$ be the function from $U$ to $U$ given by $f(u_\sigma) = u_\tau$ where $\tau(n) = \sigma(n+1)$ for all $n$. (In other 
words, $f$ deletes the first entry of $\sigma$ and moves all other entries one place to the left.) We let $G$ be the group of all order-preserving permutations of $U$ under the induced ordering which 
preserve the function $f$.

Now let $U$ be the set of atoms in a model $\mathcal M$ of Fraenkel--Mostowski set theory with choice, and let $\mathcal N$ be the Fraenkel--Mostowski submodel obtained from $U$, with the group $G$ and 
normal filter of subgroups generated by the stabilizers of finite subsets of $\mathcal Q$.

It is clear that $f$ is a non-identity epimorphism of $\mathcal Q$, and so in $\mathcal N$, ${\rm Epi}(U, \le)$ is non-trivial. We just have to show that ${\rm Emb}(U, \le)$ is trivial. Suppose for a 
contradiction that $g \in {\mathcal N}$ is a non-identity order-embedding of $U$, and let $u \in U$ be moved by $g$. Then as $g$ is order-preserving, $g^n(u)$ are all distinct and so $U$ has a countably 
infinite subset in $\mathcal N$. Let this be supported by the finite subset $X$ of $U$. We suppose that $u_{0^\omega}$ (where $0^\omega$ is the all-zero sequence in $\mathcal Q$, identified as the unique 
member of ${\mathbb Q}^0$) lies in $X$, as it is fixed by every automorphism. Observe that every member of $U$ is mapped to $u_{0^\omega}$ by some power of $f$. Consequently, we may assume that $X$ is 
closed under $f$. If $u_h \not \in X$, then $u_h$ lies in an infinite orbit of $G_X$. Let $n$ be greatest such that $u_h$ agrees with some member of $X$ for the first $n$ places. Then for some $q$ and $r$ 
(or allowing $q = -\infty$ and/or $r = \infty$ if necessary) $q$ is the greatest rational, and $r$ is the least, such that some member of $X$ agrees with $h$ for the first $n$ places, and then has 
$q < h(n) < r$. Clearly the automorphism group of $\mathcal Q$ acts transitively on sequences agreeing with $h$ except on the $n$th place, and with $n$th entry between $q$ and $r$, and so the orbit of $h$ 
is infinite as claimed. Therefore no countably infinite subset of $U$ can have finite support.    
\end{proof}  

We remark that in this model, $|U|$ lies in $\Delta \setminus \Delta_5$ in the notation of \cite{Truss}. ($\Delta$ is the class of cardinals of sets admitting no injection to a proper subset, and $\Delta_5$ 
is the class of cardinals of sets not admitting a surjection to a proper superset.)

\section{Higher cardinalities}

We return to the AC situation, but now at higher cardinalities. In \cite{Droste} it was shown how to construct dense chains in any uncountable cardinality $\kappa$ which are rigid but admit many embeddings. 
These ideas can be adapted to consider epimorphisms too. We give two constructions of $(X, <)$, as in \cite{Droste}, first the basic one which is rigid, and demonstrate that it is also epimorphism-rigid, 
and the other which is still epimorphism-rigid, but admits embeddings into every interval (so is far from embedding-rigid). The chains are exactly the same as before, but for completeness we give an outline 
of their construction, concentrating on establishing epimorphism-rigidity. For ease we just deal with the case of $\kappa$ regular.

\begin{theorem} \label{5.1} For any uncountable regular cardinal $\kappa$ there is a dense epimorphism-rigid chain $(X, <)$ without endpoints of cardinality $\kappa$.  \end{theorem}

\begin{proof} We consider ${\mathbb L}_\kappa = \kappa . {\mathbb Q}$ under the lexicographic order, $\kappa$ copies of $\mathbb Q$. If $S \subseteq \kappa \setminus \{0\}$ is stationary, we also write 
${\mathbb L}_{\kappa, S}$ for the union of ${\mathbb L}_\kappa$ with $\{(\alpha, -\infty): \alpha \in S\}$, also ordered lexicographically. Thus a new point is added immediately to the left of the 
$\alpha$th copy of $\mathbb Q$, for each $\alpha$ in $S$. We remark that there is a clear notion of `copy of $\mathbb Q$' in ${\mathbb L}_{\kappa, S}$, namely, a point lies in a copy of $\mathbb Q$ if it
has the form $(\alpha, q)$ for some $\alpha < \kappa$, $q \in {\mathbb Q}$. 

We start with a family $\mathcal S$ of $\kappa$ pairwise disjoint stationary subsets of $\kappa$, which is known to exist by \cite{Solovay}, and write $\mathcal S$ as the disjoint union of $\aleph_0$ 
pairwise disjoint sets ${\mathcal S}_n$ of cardinality $\kappa$. 

We build an increasing sequence of dense chains without endpoints $X_0 \subseteq X_1 \subseteq X_2 \subseteq \ldots$ and the final chain is $X = \bigcup_{n \in \omega}X_n$. We 
start with $X_0 = {\mathbb L}_\kappa$. To perform the extension from $X_n$ to $X_{n+1}$ we require many stationary sets to `encode' certain cuts, and to stop there being unwanted epimorphisms, and use the 
families ${\mathcal S}_n$ just introduced. Assume that $X_n$ has been defined, and that there is a notion of `copy of $\mathbb Q$' in $X_n$. For each point $x$ of $X_n$ which lies in a copy of $\mathbb Q$,
we adjoin immediately to its left a set of the form ${\mathbb L}_{\kappa, S}$, where the sets $S$ are distinct (and hence disjoint) members of ${\mathcal S}_n$. Note that copies of $\mathbb Q$ which exist in
$X_n$ are thereby destroyed since their members $x$ now all have cofinality $\kappa$ (meaning that $\kappa$ is the least cardinality of a well-ordered cofinal subset of $(-\infty, x)$), but lots more 
copies of $\mathbb Q$ are added, which in turn are destroyed in passing to $X_{n+2}$ and so on. The idea is that the stationary set $S = S_x$ such that ${\mathbb L}_{\kappa, S}$ is added immediately to the 
left of $x$ acts as a `code' for $x$, and that this is sufficiently robust to be recognizable even in $X$. Note that a point gets encoded if and only if at some stage it lies in a copy of $\mathbb Q$. 
Clearly such points are dense in $X$. They have cofinality $\kappa$ also in $X$, since no more new points are ever added immediately to the left of $x$. All other points arise at some intermediate stage, and 
since they do not lie in any copy of $\mathbb Q$, must have the form $(\alpha, -\infty)$ for some $\alpha < \kappa$; these have cofinality $< \kappa$, which they also retain in $X$.

It was shown in \cite{Droste} that $X$ is (automorphism-)rigid. We adapt that argument to show that it is also epimorphism-rigid. The key point is that any epimorphism $f$ extends (uniquely) to the 
order-completion $\overline X$, where it is continuous. These statements are proved as in Lemmas \ref{2.7}(i) and \ref{2.6}. Now suppose for a contradiction that $f$ is a non-identity epimorphism of $X$, and
let $x'$ be moved by $f$. Let $y' = f(x')$, so that $y' \neq x'$. Suppose that $x' < y'$ (a similar argument applying if $x' > y'$). Since the set of coded points is dense, there is a coded point $y$ such 
that $x' < y < y'$. As $f$ is surjective, $f^{-1}(y)$ is non-empty, and is clearly bounded below. Let $x$ be the infimum of $f^{-1}(y)$. By continuity of $f$, $f(x) = y$, and also for all $z < x$, 
$f(z) < y$. Also, $x < y$. Let $x$ have cofinality $\lambda$. Then the image of an increasing sequence witnessing this is cofinal in $(-\infty, y)$, so as $y$ has cofinality $\kappa$, it follows that 
$\lambda = \kappa$, and hence that $x$ has cofinality $\kappa$. To see that $x$ is also coded, it suffices to note that $x \in X$. Suppose not, for a contradiction, and let $x_\alpha$ for $\alpha < \kappa$
be a strictly increasing sequence of points of $X$ with supremum equal to $x$. Now $X = \bigcup_{n \in {\mathbb N}}X_n$, so as $\kappa > \omega$ is regular, we may pass to a subsequence of $(x_\alpha)$
all of whose entries lie in the same $X_n$, and we take minimal such $n$. Since any increasing $\kappa$-sequence in ${\mathbb L}_\kappa$ is unbounded, $n > 0$. Since $x_\alpha \in X_n \setminus X_{n-1}$, 
$x_\alpha$ lies in a copy of some ${\mathbb L}_{\kappa, S}$ immediately to the left of a unique $x_\alpha'$ in a copy of $\mathbb Q$ in $X_{n-1}$. Since ${\mathbb L}_{\kappa, S}$ has no bounded increasing 
$\kappa$-sequence, $\kappa$ of the $x_\alpha'$ are distinct, and this gives a strictly increasing sequence of points of $X_{n-1}$ having $x$ as supremum, contrary to minimality of $n$. Hence $x$ lies in $X$ 
and is coded.

Let $L_x$ be the subset of the order-completion of ${\mathbb L}_{\kappa, S_x}$ of `infinite' points $x_\alpha = (\alpha, -\infty)$ for $\alpha \in \kappa \setminus \{0\}$, where this is the set that was 
added immediately to the left of $x$ at stage $n$ in the construction. Then $L_x$ is a closed unbounded subset of $(-\infty, x) \cap {\overline X}_{n+1}$, and it is also closed unbounded in 
$(-\infty, x) \cap {\overline X}$ since no cut immediately to the left of any of its points is ever realized. Moreover, $L_x \cap X = S_x$. The same remarks apply to $y$ and points $y_\alpha$. In 
particular, $L_y \cap X = S_y$. We can now see that $C = \{\alpha \in \kappa: f(x_\alpha) = y_\alpha\}$ is a closed unbounded subset of $\kappa$. Closure follows from the facts that $L_x$ and $L_y$ are 
closed, and that $f$ is continuous. For unboundedness, consider any $\alpha \in \kappa$. We form a sequence $\alpha = \alpha_0 \le \alpha_1 \le \alpha_2 \le \ldots$ in $\kappa$ thus. Suppose that $\alpha_n$ 
has been chosen, $n$ even. Then $f(x_{\alpha_n}) < y$ and so $f(x_{\alpha_n}) \le y_{\alpha_{n+1}}$ for some $\alpha_{n+1} \ge \alpha_n$. If $\alpha_n$ has been chosen, $n$ odd, $f^{-1}(y_{\alpha_n})$ is 
bounded below $x$, and so there is some $\alpha_{n+1} \ge \alpha_n$ such that $f^{-1}(y_{\alpha_n}) \le x_{\alpha_{n+1}}$. Hence $y_{\alpha_n} \le f(x_{\alpha_{n+1}})$. Let $\beta$ be the supremum of the 
increasing sequence $(\alpha_n: n \in \omega)$. Thus $f(x_{\alpha_0}) \le y_{\alpha_1} \le f(x_{\alpha_2}) \le y_{\alpha_3} \le \ldots$, so 
$\sup_{n \in {\mathbb N}}f(x_{\alpha_n}) = \sup_{n \in {\mathbb N}} y_{\alpha_n}$. By continuity of $f$, $f(x_\beta) = f(\sup_{n \in {\mathbb N}}x_{\alpha_n}) = \sup_{n \in {\mathbb N}}f(x_{\alpha_n}) =
\sup_{n \in {\mathbb N}}y_{\alpha_n} = y_\beta$, so $\beta \in C$.

Since $S_x$ is stationary, it intersects $C$. Let $\alpha \in S_x \cap C$. Then $x_\alpha \in X$ since $\alpha \in S_x$, and $f(x_\alpha) = y_\alpha$ since $\alpha \in C$. Therefore $y_\alpha \in X$, and
so $\alpha \in S_y$. This contradicts the assumption that $S_x \cap S_y = \emptyset$.     \end{proof}

We have been unable to determine whether the chain constructed in this theorem is also embedding-rigid, the main problem being that there is no reason why embeddings should be continuous. 

The second result is a strengthening of the first in which it is very much {\em not} embedding-rigid.

\begin{theorem} \label{5.2} For any uncountable regular cardinal $\kappa$ there is a dense chain $(X, <)$ without endpoints of cardinality $\kappa$ that is epimorphism-rigid but which embeds into any 
non-empty open interval.  \end{theorem}

\begin{proof} We modify the basic construction of \ref{5.1}. First we work instead with finite sequences ${\underline S} = (S_0, \ldots, S_{m-1})$ of stationary sets rather than single ones, and write 
${\mathbb L}_{\kappa, {\underline S}}$ for the concatenation ${\mathbb L}_{S_0}\,^\wedge \ldots ^\wedge {\mathbb L}_{S_{m-1}}$. If $\mathcal A$ is a family of $\kappa$ finite non-empty sequences of
stationary subsets of $\kappa \setminus \{0\}$, whose final entries are pairwise disjoint, divided into $\aleph_0$ disjoint subsets ${\mathcal A}_n$ of size $\kappa$, we let 
$X_{\mathcal A} = \bigcup_{n \in \omega}X_n$ where $X_0 = {\mathbb L}_\kappa$, and $X_{n+1}$ is obtained from $X_n$ by inserting sets of the form ${\mathbb L}_{\kappa, {\underline S}}$ for 
$s \in {\mathcal A}_n$ immediately to the left of all points of $X_n$ lying in a copy of $\mathbb Q$. This is exactly as in \ref{5.1}, except that `larger' sets are inserted at each stage. The pairwise 
disjointness of the final entries guarantees that a dense set of points is again `coded' by stationary sets.

In order to construct a chain which also admits many embeddings, we modify this construction. We let $\mathcal S$ be a family of $\kappa$ pairwise disjoint stationary subsets of $\kappa \setminus \{0\}$, 
and let $I_0 \subset I_1 \subset I_2 \subset \ldots$ be sets of cardinality $\kappa$ such that each $I_{n+1} \setminus I_n$ also has cardinality $\kappa$. If $\Sigma$ is the family of all finite sequences 
$(i_0, i_1, \ldots , i_{n-1})$ such that $i_j \in I_j$ for each $j$, then $|\kappa \times \Sigma| = \kappa$, so we may let ${\mathcal S} = \{S_{\alpha \sigma}: \alpha \in \kappa, \sigma \in \Sigma\}$. For 
each $\sigma \in \Sigma$, we let ${\mathcal A}_\sigma$ be the family of all sequences $(S_{\alpha ()}, S_{\alpha (\sigma(0))}, S_{\alpha (\sigma(0) \sigma(1))}, \ldots, S_{\alpha \sigma})$ for 
$\alpha < \kappa$. A key point is that the stationary sets which arise as the final entry of some such sequence are pairwise disjoint. Using the method described in the previous paragraph, we let 
$X_\sigma = X_{{\mathcal A}_\sigma}$ for each $\sigma \in \Sigma$. 

Now ${\mathbb L}_{(S_{\alpha ()}, S_{\alpha (\sigma(0))}, S_{\alpha (\sigma(0) \sigma(1))}, \ldots, S_{\alpha \sigma})}$ embeds naturally into \newline
${\mathbb L}_{(S_{\alpha ()}, S_{\alpha (\sigma(0))}, S_{\alpha (\sigma(0) \sigma(1))}, \ldots, S_{\alpha \sigma}, S_{\alpha, \sigma^\wedge\langle i\rangle})}$, and we want to ensure that there is a
corresponding embedding $f_{\sigma i}$ of $X_\sigma$ into $X_{\sigma^\wedge\langle i\rangle}$ (which will definitely not be continuous), and to achieve this, we have to choose $X_\sigma$ rather more 
carefully, using induction on the length of $\sigma$. For the empty sequence, the construction is as usual. Now assuming that $X_\sigma$ has been chosen, embed each
${\mathbb L}_{(S_{\alpha ()}, S_{\alpha (\sigma(0))}, S_{\alpha (\sigma(0) \sigma(1))}, \ldots, S_{\alpha \sigma})}$ into  
${\mathbb L}_{(S_{\beta ()}, S_{\beta (\sigma(0))}, S_{\beta (\sigma(0) \sigma(1))}, \ldots, S_{\beta \sigma}, S_{\beta, \sigma^\wedge\langle i\rangle})}$ for some $\beta$, leaving $\kappa$ values of 
$\beta$ untouched, and use these `spare' copies to fill in all the additional points which have now arisen and which need to be filled during the construction. By composing, we get embeddings 
$f_{\sigma,\tau}$ of $X_\sigma$ into $X_{\sigma ^\wedge \tau}$.

The final chain $X$ is the disjoint union of all the $X_\sigma$ for $\sigma \in \Sigma$, and the relation between points lying in distinct $X_\sigma$s is determined by inserting the whole of 
$X_{\langle i\rangle ^\wedge \sigma}$ into a certain irrational cut of $X_\sigma$, where by `irrational cut' we mean, an irrational cut in a copy of $\mathbb Q$ that appears at some stage. The full details 
as given in \cite{Droste} are omitted. The main points are that the fact that this $X$ is epimorphism-rigid is proved just as in the previous theorem, since the set of encoded points is dense, and each 
encoded point is coded by a distinct stationary set, and again as in \cite{Droste} one checks using the embeddings $f_{\sigma,\tau}$ that the whole of $X$ can be embedded into every non-trivial interval.
\end{proof}


\end{document}